\documentclass[12pt]{article}%

\usepackage{graphicx}
\usepackage{graphics}

\usepackage{color}
\usepackage{amsmath}
\usepackage{amsfonts}
\usepackage{amssymb}
\usepackage{enumerate}
\usepackage[normalem]{ulem}

\definecolor{PineGreen}{cmyk}{0.9,0,1,0.40}
\def\Green{\color{black}}

\newtheorem{theorem}{Theorem}[section]

\newtheorem{corollary}[theorem]{Corollary}

\newtheorem{definition}[theorem]{Definition}

\newtheorem{proposition}[theorem]{Proposition}

\newtheorem{remark}[theorem]{Remark}

\newenvironment{proof}[1][Proof]{\textbf{#1.} }{\ \rule{0.5em}{0.5em}}

\newcommand{\E}{{\rm \bf E}}
\newcommand{\Leb}{{\rm Leb}}

\newcommand{\prob}{{\rm \bf P}}

\newcommand{\dN}{{\mathbb N}}

\newcommand{\dR}{{\mathbb R}}

\newcommand{\rmd}{{\rm d}}

\newcommand{\calF}{{\cal F}}

\newcommand{\calM}{{\cal M}}

\newcommand{\calS}{{\cal S}}

\newcommand{\calU}{{\cal U}}

\newcommand{\ep}{\varepsilon}
\def\ind{{\bf 1}\hskip-2.5pt{\rm l}}

\title{Guessing a Random Function
	and Repeated Games in Continuous Time\thanks{Solan acknowledges the support of the Israel Science Foundation, grant \#211/22.
}}
\author{Catherine Rainer%
	\thanks{Univ Brest, UMR CNRS 6205, 6, avenue Victor-le-Gorgeu, B.P. 809, 29285 Brest cedex, France.
		e-mail: Catherine.Rainer@univ-brest.fr.} 
	\ and Eilon Solan%
	\thanks{The School of Mathematical Sciences, Tel Aviv
		University, Tel Aviv 6997800, Israel. e-mail: eilons@post.tau.ac.il.} }
\date{\today}

\begin{document}
	
\maketitle

\begin{abstract}
We study a game
where one player selects a random function,
and the other has to guess that function,
and show that with high probability the second player can correctly guess most of the random function.
We apply this analysis to continuous-time repeated games
played with mixed strategies with delay,
identify good responses of a player to any profile of her opponents,
and show that the minmax value coincides with the minmax value in pure strategies of the one-shot game.
\end{abstract}

\bigskip


%\bigskip

%\noindent\textbf{JEL Classification:} C73.

\bigskip
\noindent\textbf{Declarations of interest:} none.
	
\section{Introduction}

Repeated games in continuous time are studied since the late 80's. 
Various authors proposed frameworks to analyze these games
and compared the prediction of the continuous-time model to those of the corresponding discrete-time model,
see, e.g., Simon and Stinchecombe (1989), Bergin and MacLeod (1993), and Al\'os-Ferrer and Kern (2015).
%Bergin and Mcleod (1993) also proved a Folk Theorem relative to the minmax value in \emph{pure} actions.

A class of games that contains repeated games in continuous time is that of differential games, which has been widely studied since the seminal works of Friedman (1971) and Varaya (1967), see, e.g., Bressan (2011) and the references therein.
It is well known that under the so-called Isaac's condition (and suitable technical regularity assumptions on the data),
differential games admit an equilibrium (potentially not subgame perfect), when players use pure non-anticipative strategies, see Evans and Souganidis (1984) in the zero-sum case,
and Kleimenov (1993) and Buckdahn, Cardaliaguet, and Rainer (2004) for the non-zero-sum case.
Cardaliaguet (2007) proved the existence of the value for zero-sum games with asymmetric information, using random strategies with delay.
As a by-product of this last result, 
one deduces the existence of the value and its characterization in repeated games in continuous time 
if and only if Isaac's condition holds,
even if players are allowed to play randomly.

In this paper, we identify a good response of a player in a repeated game in continuous time to any given profile of mixed strategies with delay of her opponents, namely, a strategy that guarantees a payoff no less than the minmax value of the player in pure strategies, up to some small error term.
To highlight the main idea of the construction, we concentrate on
the continuous-time Matching Pennies game.
The reader will have no difficulty to extend the construction to all multiplayer repeated games in continuous time.
We hope that our technique will pave the way to an analogous construction for differential games.

%	For strategies in continuous-time games, we adopt the concept of nonanticipative strategies with delay:
%	a pure strategy of player~$i$ is given by an increasing sequence $(t_k)_{k \in \dN}$ of reals that increases to infinity
%	and a rule that indicates a pure play for player~$i$ in the time interval $[t_k,t_{k+1})$ for every play of all players in the time interval $[0,t_k)$ and every $k \in \dN$.
%	Thus, for each $k \in \dN$, the player ``wakes up'' at time $t_k$, observes the past play,
%	determines her play until time $t_{k+1}$, and ``falls asleep'' again.
%	As usual, a mixed strategy is a probability distribution over pure strategies.
	
The driving force behind this result can be illustrated by the following game, which we term ``guessing a random function'', and is played between two players, say, Aqua and Bard. 
In this game, Aqua chooses a random function $f : [0,1] \to \{a,b\}$.
Formally, Aqua chooses a probability space $(\Omega,\calF,\prob)$ and a measurable function $f : \Omega \times [0,1] \to \{a,b\}$.
	Nature then chooses $\omega \in \Omega$ according to $\prob$,
	and thereby generates a random function $f(\omega,\cdot) : [0,1] \to \{a,b\}$.
Bard's goal is to guess $f(\omega,\cdot)$.
	She does so in the following way.
First, Bard selects a partition $0 = t_0 < t_1 < \dots < t_K = 1$.
	The rest of the game is divided into $K$ rounds.
	In round $k$, $k=0,1,\ldots,K-1$,
	Bard chooses a function $g : [t_k,t_{k+1})$ and is then told the restriction of $f$ to $[t_k,t_{k+1})$.
	The payoff, which Bard tries to maximize and Aqua to minimize,
	is the Lebesgue measure of $\{t \in [0,1) \colon f(t) = g(t)\}$.

In Section~\ref{section:1}, for every choice of Aqua we construct a response of Bard that correctly guesses most of the random section $f(\omega,\cdot)$.
In particular, 
the minmax value of the ``guessing a random function'' game is 1.
The reason is that the measurability of $f$, 
	together with the fact that the range of $f$ is finite,
	imply that if the differences $(t_{k+1}-t_k)_{k \in \dN}$ are sufficiently small,
	then with high probability, the knowledge of $f(\omega,\cdot)$ on the interval $[t_k,t_k+\ep_k)$ for some small $\ep_k > 0$ reveals 
	the value of $f(\omega,\cdot)$ on most of the interval $[t_k+\ep_k,t_{k+1})$.
	In other words, a mixed non-anticipative strategy with delay is nearly analogous to a pure strategy in discrete time, rather than to a mixed strategy in discrete time.
	
In Section~\ref{section:2} we relate the results on the ``guessing a random function'' game to those in the Matching Pennies game in continuous time.
We explain through this simple example that when the players' strategy set is the set of mixed strategies with delay, the minmax value of a player in a repeated game in continuous time is her minmax value in pure strategies of the one-shot game.
In particular, 
in continuous-time games, 
using mixed strategies with delay 
does not increase the amount a player can defend above the minmax value in pure strategies.

%%%%%%%%%%%%%%%%%%%%%%%%%%%%%%%%%%%%%
\section{The Guessing a Random Function Game}
\label{section:1}
	
	In this section we study the following two-player zero-sum game, where one player randomly selects a measurable function defined on $[0,1)$ with values in a set of two elements,
	and the other player has to guess that function.
	
	Formally, the game is played between two players, Aqua and Bard,
	and evolves as follows.
	\begin{enumerate}
		\item
		Aqua selects a probability space $(\Omega,\calF,\prob)$ and a jointly-measurable function $f : \Omega \times [0,1) \to \{a,b\}$.
		For each $(\omega,t) \in \Omega \times [0,1]$ denote $f_\omega(t) = f(\omega,t)$.
		The choices of Aqua are told to Bard.
		\item
		Nature selects $\omega \in \Omega$ according to $\prob$.
  Bard is not informed about the choice of $\omega$.
		\item
		Bard selects an increasing sequence $0 = t_0 < t_1 < \cdots < t_K = 1$.
		She then selects a function $g : [0,1) \to \{a,b\}$ as follows.
		For each $k=0,1,\ldots,K-1$,
		\begin{itemize}
			\item Bard selects the restriction of $g$ to $[t_k,t_{k+1})$.
			\item Nature announces the restriction of $f_\omega$ to $[t_k,t_{k+1})$.
		\end{itemize}
	\end{enumerate}
	
	Bard's {\Green expected payoff is
	\[ \gamma^{\rm Bard}(f,g):=\E\left[\int_0^1\ind_{\{ f(t)=g(t)\}}\rmd t\right]=\prob\otimes \Leb\left(\left\{ (t,\omega)\in[0,1]\times\Omega, f(\omega,t)=g(\omega,t)\right\}\right), \]}
 where $\lambda$ stands for the Lebesgue measure on $\dR$.
	
	In words, Aqua randomly selects a measurable function from $[0,1)$ to $\{a,b\}$,
	and Bard guesses this function in pieces:
	she selects a partition of the unit interval,
	and then, for each interval in this partition,
	she guesses the value of the function at all points in the interval after observing the value of the function in all previous intervals.
	
	We will prove that Bard's minmax value is 1:
	for every choice of Aqua, Bard has a response that guarantees that she guesses $f_\omega$ correctly with arbitrarily high probability, for an arbitrarily high proportion of $t \in [0,1)$.
%Since Bard moves after Aqua, 
%and since 1 is the maximal payoff Bard can obtain, 
%1 is also the value of the game.
	
Fix then a choice of Aqua,
namely, a probability space $(\Omega,\calF,\prob)$ and a jointly-measurable function $f : \Omega \times [0,1) \to \{a,b\}$.
	Fix also $\ep \in (0,1)$.
	
	Since $f$ is jointly measurable,
	the set $\Omega_1 \subseteq \Omega$ of all $\omega \in \Omega$ for which 
	the section $f_\omega : [0,1) \to \{a,b\}$ is measurable
	has measure 1.
	By Lusin's Theorem,
	for every $\omega \in \Omega_1$ 
	there exists a compact set $C_\omega \subseteq [0,1)$
	such that $f$ is continuous on $C_\omega$ and $\prob(C_\omega) > 1-\ep$.
	Since $C_\omega$ is a compact subset of $[0,1)$, it is a union of finitely many closed intervals.
	Since $f_\omega$ is continuous on $C_\omega$,
	$f_\omega$ is constant on each of those closed subintervals.
	
	If we divide the interval $[0,1)$ into sufficiently many subintervals of equal length, most of them will be included in $C_\omega$.
	Formally,
	there is $n_\omega \in \dN$ such that 
	\[ \frac{1}{2^{n_\omega}} \cdot \# \left\{k \in \{0,1,\ldots,2^{n_\omega}-1\} \colon \left[\frac{k}{2^{n_\omega}},\frac{k+1}{2^{n_\omega}}\right] \subseteq C_\omega\right\} > 1-2\ep. \]
	Let $n_* \in \dN$ be sufficiently large such that
	\[ \prob\bigl( \Omega_1 \cap \{n_\omega \leq n_*\} \bigr) > 1-\ep. \]
	This definition implies that for every $\omega \in \Omega_1 \cap \{n_\omega \leq n_*\}$
	and every $k$ such that $[\frac{k}{2^{n_\omega}},\frac{k+1}{2^{n_\omega}}] \subseteq C_\omega$,
	the section $f_w$ is constant on $[\frac{k}{2^{n_\omega}},\frac{k+1}{2^{n_\omega}}]$.
	
	We are now ready to define Bard's selection.
	Let $K = 2^{n_* + 1}$,
	$t_{2k} = \frac{k}{2^{n_*}}$,
	and $t_{2k+1} = \frac{k+\ep}{2^{n_*}}$ for each $k \in \{0,1,\ldots,2^{n_*}-1\}$.
	The function $g$ is defined as follows: 
	For each $k \in \{0,1,\ldots,2^{n_*}-1\}$,
	\begin{itemize}
		\item 
		on $[t_{2k},t_{2k+1})$ the function $g$ is defined arbitrarily;
		\item 
		on $[t_{2k+1},t_{2k+2})$
		the function $g$ is constant,
		and is equal to the value that is more common on $[t_{2k},t_{2k+1})$.
		That is, for each $t \in [t_{2k+1},t_{2k+2})$, $g(t) = a$ if $\lambda(\{s \in [t_{2k},t_{2k+1}) \colon g(s) = a\}) > \frac{\ep}{2^{n_*+1}}$, and $g(t) = b$ otherwise.
	\end{itemize}
	
	For $t \in \bigcup_{k=0}^{2^{n_*}-1} [t_{2k+1},t_{2k+2})$,
	Bard's guess is correct provided that
	(a) $\omega \in\Omega_1$,
	(b) $n_\omega \leq n_*$, and
	(c) $[\frac{k}{2^{n_\omega}},\frac{k+1}{2^{n_\omega}}] \subseteq C_\omega$.
	Therefore, Bard's payoff when she responds with this strategy is at least $(1-\ep)(1-2\ep)$.

 The function $g$ depends on $f$, and is itself a random function.
 We denote it by $\alpha_\ep(f) : \Omega \times [0,1] \to \{a,b\}$.
By construction, $f$ and $\alpha_{\ep}(f)
$ coincide on a set whose measure is larger than $1-3\ep$.
	
The following proposition summarizes the discussion above.

\begin{proposition}
		\label{minmax for choosing a random}
		For every function $f:\Omega\times[0,1]\to\{a,b\}$ and every $\ep>0$, the function $\alpha_\ep(f)$ satisfies
		\[ \prob\bigl( \Leb\left(\{t\in[0,1],f(t)=\alpha_\ep(f)(t)\}\right)\geq 1-2\ep\bigr)\geq 1-\ep.\]
\end{proposition}

Proposition~\ref{minmax for choosing a random} implies that for every $f:\Omega\times[0,1]\to\{a,b\}$ and every $\ep>0$,
		\[ \gamma^{\rm Bard}(f,\alpha_\ep(f))\geq 1-3\ep. \]
This delivers the following conclusion.

\begin{corollary}
The value of the ``choosing a random function'' game is 1.
\end{corollary}

\section{The Continuous-Time Matching Pennies Game}
	\label{section:2}
	
	The continuous-time Matching Pennies game is the continuous-time repeated game where the set of players is $I = \{{\rm Aqua}, {\rm Bard}\}$, 
	the sets of actions of the players are $A^{\rm Aqua} = A^{\rm Bard} = \{a,b\}$,
	the players' common discount rate is $r > 0$,
	and the instantaneous payoff functions are
	\[
	g^{\rm Aqua}(a^{\rm Aqua},a^{\rm Bard}) = \left\{
	\begin{array}{lll}
		1, & & a^{\rm Aqua} \neq a^{\rm Bard},\\
		0, &  & a^{\rm Aqua} = a^{\rm Bard},
	\end{array}
	\right.
	\ \ \ \ \ 
	g^{\rm Bard}(a^{\rm Aqua},a^{\rm Bard}) = \left\{
	\begin{array}{lll}
		0, & & a^{\rm Aqua} \neq a^{\rm Bard},\\
		1, &  & a^{\rm Aqua} = a^{\rm Bard}.
	\end{array}
	\right.
	\]
	
	Thus, at every time instance,
	Bard wants to match Aqua's choice,
	while Aqua wants to mismatch Bard's choice.
	
In this section, 
given any strategy of Aqua,
we construct a response of Bard that achieves an expected payoff close to 1.
In particular, the minmax value of the continuous-time Matching Pennies game is 1
(and, symmetrically, the maxmin value is 0).
%for every $\ep < \frac{1}{2}$ the game has no $\ep$-equilibrium in non-anticipative strategies with delay.

The analysis we will conduct is not special to the Matching Pennies game.
Analogous arguments show that in all repeated games in continuous time with finitely many players and actions,
when each player's strategy set is the set of mixed strategies with delay, the minmax value of a player is her minmax value in pure strategies of the one-shot game.
In particular, 
in continuous-time games, 
using mixed strategies with delay 
does not increase the amount a player can defend above the minmax value in pure strategies.

We start by defining the class of strategies we will work with, namely, non-anticipative strategies with delay.
	
	\begin{definition}[controls]
		Let $i \in I$.
		A \emph{control of player~$i$} is a measurable function $u^i : \dR_+ \to A^i$.
		The set of all controls of player~$i$ is denoted by $\calU^i$.
	\end{definition}
	
	We endow $\calU^i$ with the topology defined by the distance 
	$d(u^i,\widetilde u^i)=\int_0^\infty re^{-rt} \mathbf{1}_{\{u^i_t \neq \widetilde u^i_t\}} \rmd t$ and denote by ${\cal B}^i$ the associated $\sigma$-algebra on ${\cal U}^i$.
 
	A pair of controls $u = (u^i)_{i \in I} \in \calU = \prod_{i \in I} \calU^i$ indicates the actions chosen by the players in every time instance.
	The discounted payoff of player~$i$ under the pair of controls $u \in \calU$ is
	\[ \gamma^i(u) := \int_{0}^\infty re^{-rt} g^i(u(t)) \rmd t. \]
	
	\begin{definition}[strategies]
		\label{def:strategies}
		Let $i \in I$.
		A \emph{strategy} of player~$i$ is a measurable function $\sigma^i : \calU^{-i} \to \calU^i$, such that  
		there exists an increasing sequence of reals $0=t_0 < t_1 < \cdots$ that converges to $\infty$
		and such that
		\begin{equation}
			\label{equ:alpha}
			\left( u^{-i}=\widetilde u^{-i}\;  \lambda\mbox{-a.s. on }[0,t_n]\right) \Longrightarrow \left(\sigma^i(u^{-i})=\sigma^i(\widetilde u^{-i})\;  \lambda\mbox{-a.s. on }[0,t_{n+1}]\right), \ \ \ \forall n \in \dN.
		\end{equation}
		The set of all strategies of player~$i$ is denoted by $\calS^i$.
	\end{definition}
	In words,
	a strategy determines at time 0 how the player will play in the time interval $[0,t_1)$,
	at time $t_1$ it determines how the player will play in the time interval $[t_1,t_2)$ based on the other player's play in the time interval $[t_0,t_1)$, and so on.
	Every profile $\sigma$ of strategies induces a unique play, denoted by $u_\sigma$.
	
	To play well in the one-shot version of the Matching Pennies game,
	the players have to select their actions randomly.
	The same applies to the continuous-time version of the game.
	We therefore introduce the concept of mixed strategies, 
	which is adapted from Aumann (1964).
	Roughly, a mixed strategy is a probability space together with a function that assigns a strategy to each point in that space.
	To play a mixed strategy, an element $\omega$ in the probability space is chosen randomly, and the player follows the strategy that corresponds to $\omega$.
	
{\Green	\begin{definition}[mixed strategies]
\label{def:mixed:strategies}
Let $i \in I$.
A \emph{mixed strategy} of player~$i$ is given by a 
probability space $(\Omega^i,\calF^i,\prob^i)$ and a measurable function $\mu^i:\Omega^i\times\calU^{-i}\to \calU^i$,
such that $\mu^i(\omega^i,\cdot)\in \calS^i$ for $\prob^i$-a.e.~$\omega^i\in\Omega^i$.
The set of all mixed strategies of player $i$ is denoted by $\calM^i$, and $\calM=\prod_{i\in I}\calM^i$ is the set of mixed strategy pairs.
			\end{definition}

Given a pair 
of mixed strategies, 
we denote by $(\Omega,\calF,\prob)=(\prod_{i\in I}\Omega^i,\bigotimes_{i\in I}\calF^i,\bigotimes_{i\in I}\prob^i)$ the associated product space. We identify all random variable $X$ on $(\Omega^i,\calF^i)$ for some $i\in I$ with a random variable defined on $(\Omega,\calF)$ by setting $X(\omega)=X(\omega^i)$ for $\omega=(\omega^j)_{j\in I}$.

\begin{remark}
In the literature on continuous time zero-sum games with asymmetric information, one encounters random strategies with delay, which are mixed strategies, in the sense we  defined above, but where, for each $i\in I$, all deterministic strategies building the support of $\mu^i$ have the same deterministic time grid, 
see, e.g., Cardaliaguet and Rainer (2009). 
The reason one uses this simpler notion is that it is sufficient for the  problem under consideration: to ensure the existence of a value, it is crucial that the players may play randomly, but there is no need to use random time grids. 
In contrast, in Definition~\ref{def:mixed:strategies} we did not require that for every $\omega^i \in \Omega^i$,
 $\mu^i(\omega^i)$ uses the same time grid,
 so that our result will be valid for a more general class of strategies.
\end{remark}
	
\begin{proposition}
Every pair $\mu = (\mu^i)_{i \in I}$ of mixed strategies
induces a probability distribution over  plays: there exists a probability distribution $\prob_\mu$,
such that,  the canonical process $u$ on $\calU$ satisfies 
$\mu^i(u^{-i})=u^i$ for all $i\in I$, $\prob_\mu$-a.s., a.e.~on $\dR_+$.
	\end{proposition}

\begin{proof} The proof is adapted from Cardaliaguet and Rainer (2009).
Let $\mu=(\mu^i)_{i\in I}\in\calM$ be a pair of mixed strategies. 
For each $i\in I$, denote by $(t^i_k)_{k\in\dN}$ the time grid associated to $\mu^i$. 
Define recursively the smallest common grid for all players as follows:
	\begin{itemize}
		\item  $t_0=0$.
		\item If $t_n$ is already defined for some $n\in\dN$, then
$t_{n+1}=\inf\{ t^i_k> t_n, i\in I,k\in \dN\}$. 
	\end{itemize} 
The sequence $(t_n)_{n\in\dN}$ is an increasing  sequence of random variables on $(\Omega,\calF)$ 
that converges to $\infty$ $\prob_\mu$-a.s.

Let us start with an arbitrary pair of controls $\underline u\in\calU$ and define recursively a random pair of controls $\widetilde u$ that will satisfy  $\mu^i(\widetilde u^{-i})=\widetilde u^i$, for each $i\in I$, $\prob$-a.s., a.e.~on $\dR_+$. 

Set $\widetilde u^i_s=\mu^i(\underline u^{-i})_s$, for all $s\in[0,t_1]$.
 By the definition of non-anticipative strategies with delay, $\prob$-a.s.~on $[0,t_1]$  the values of $\mu^i(u^{-i})$ do not depend on $u^{-i}\in\calU^{-i}$ on $[0,t_1]$. In particular, for all $u\in\calU$ that coincide with $\widetilde u$ a.e.~on $[0,t_1]$, $\mu^i(u^{-i})=u^i=\widetilde u^i$ on $[0,t_1]$.
 
Suppose that for some $n\geq 1$ the process $\widetilde u$ is already defined on $[0,t_n]$.
Set
$\overline u_s:=\widetilde u_s$ for $s\leq t_n$ and $\overline u_s:=\underline u_s$ for $s>t_n$;
then $\overline u$ is a $\calU$-valued random variable on $(\Omega,\calF,\prob)$.  
As a composition of measurable maps, for each $i\in I$, $\mu^i(\overline u^{-i}) $ is well defined and also measurable. Set now $\widetilde u^i_i:=\mu^i(\overline u^{-i})_s$ on $(t_n,t_{n+1}]$. 
Again, by the definition of non-anticipative strategies with delay, 
on $[0,t_{n+1}]$ the processes $(\mu^i(\overline u^{-i}))_{i \in I}$ depend only on their restrictions to $[0,t_n]$, 
so that, for all random pairs of controls $u\in\calU$, if $u=\widetilde u$ $\prob$-a.s., a.e.~on $[0,t_{n+1}]$, for all $i\in I$,
then
$\mu^i(u^{-i})=u^i=\widetilde u^i$ $\prob$-a.s., a.e.~on $[0,t_{n+1}]$.

Iterating the construction for all $n\in\dN$, we define the process $\widetilde u$ on $\dR$ as a fixed point of the mixed strategy $\mu$.

Finally, the law of $\widetilde u$ defines a probability $\prob_\mu$ on $\calU$, which satisfies the required relation.
\end{proof}

\bigskip

	We denote by $\E_\mu$ the expectation w.r.t.~$\prob_\mu$.
	The expected discounted payoff to player~$i$ under the mixed strategy profile $\mu$ is, then,
	\[ \gamma^i(\mu) := \E_\mu[\gamma^i(u_\sigma)]. \]}

%For $\ep\in( 0,\frac 12)$, 
%	a mixed strategy vector $\mu_*$ is an \emph{$\ep$-equilibrium} 
%	if 
%	\[ \gamma^i(\mu_*) \geq \gamma^i(\mu^i,\mu^{-i}_*) - \ep, \ \ \ \forall i \in I, \forall \mu^i \in \calM^i. \]
 
We next construct a good response for Bard against any given strategy of Aqua.

% Since the continuous-time Matching Pennies game is a two-player constant-sum game,
%a limit of $\ep$-equilibria as $\ep$ goes to 0 is a saddle point,
%and every $\ep$-equilibrium is composed of a pair of $\ep$-optimal strategies.

%	{\Green As already mentioned, the following proposition is already known, as a particular case of \cite{Cardaliaguet}. The novelty here is the constructive proof.
	\begin{proposition}
		\label{matching no eq}
%	The continuous-time Matching Pennies game admits no $\ep$-equilibrium for $\ep < \frac{1}{2}$. 
For every mixed strategy $\mu^{\rm Aqua}\in{\cal M}^{\rm Aqua}$ and for all $\ep>0$, there exists a (non-mixed) strategy $\sigma^{\rm Bard}\in\calS^{\rm Bard}$ such that $\gamma^{\rm Bard}(\mu^{\rm Aqua},\sigma^{\rm Bard}) > 1-\ep$. 
By symmetry, for every mixed strategy $\mu^{\rm Bard}\in\calM^{\rm Bard}$,	Aqua has a (non-mixed) strategy $\sigma^{\rm Aqua}$ such that $\gamma^{\rm Aqua}(\sigma^{\rm Aqua},\mu^{\rm Bard}) > 1-\ep$.
	%Since the sum of payoffs in the game is 1, 	for any $\ep < \frac{1}{2}$ an $\ep$-equilibrium in mixed strategies cannot exists.
		
	\end{proposition}

%	To prove that the continuous-time Matching Pennies game admits no $\ep$-equilibrium for $\ep < \frac{1}{2}$,
%	we will show that for every mixed strategy $\mu^{\rm Aqua}$ of Aqua and every $\ep > 0$,
%	Bard has a (non-mixed) strategy $\sigma^{\rm Bard}$ such that $\gamma^{\rm Bard}(\mu^{\rm Aqua},\sigma^{\rm Bard}) > 1-\ep$.
%	By symmetry,
%	for every mixed strategy $\mu^{\rm Bard}$ of Bard and every $\ep > 0$,
%	Aqua has a (non-mixed) strategy $\sigma^{\rm Aqua}$ such that $\gamma^{\rm Aqua}(\sigma^{\rm Aqua},\mu^{\rm Bard}) > 1-\ep$.
%	Since the sum of payoffs in the game is 1,
%	for any $\ep < \frac{1}{2}$ an $\ep$-equilibrium in mixed strategies cannot exists.

%\begin{proof}
%	Fix $\ep\in(0,\frac 12)$ and $\mu^{\rm Aqua}\in\calM^{\rm Aqua}$. We will define a strategy $\sigma^{\rm Bard}\in\calS^{\rm Bard}$ as follows:  
%%	Let $t_0=0<\ldots<t_N=\infty$ be the time grid associated to $\mu^{\rm Aqua}$. This time grid is known by Bard. Further the non anticipativity of the strategy of Aqua implies that, on $[0,t_1)$, Aqua plays a random control $u^{\rm Aqua}$, independently of the actions of Bard, who knows a time 0 its distribution. This enables her to use the strategy defined in the previous section: by Lusin's Theorem, 
%\end{proof}
	
\begin{proof}
Fix $\ep \in(0,1/2)$, and let $T=-\log(\ep)/r$,
so that $\int_T^\infty re^{-rt}\rmd t=\ep$.
Let $\mu^{\rm Aqua}\in\calM^{\rm Aqua}$ be a mixed strategy of Aqua.
		For $n\geq 1$ and $k\in\dN$, set $t^n_k=\frac{kT}{2^n}$. 	
		Let $\theta$ be the random time grid associated to the strategy $\mu^{\rm Aqua}$,
and define
		\[U^n_T=\bigcup_{k\in\{ 0,\ldots,2^n-1\}, (t_k^n,t^n_{k+1})\cap\theta=\emptyset}[t^n_k,t^n_{k+1}].\]
The sequence of sets $(U^n_T)_{n\geq 1}$ increases $\prob^{\rm Aqua}$-a.s. to $[0,T]$ as $n\nearrow\infty$. Therefore, we can find $\widehat n$ such that, for all $n\geq \widehat n$,
\[ \E^{\rm Aqua}\left[\int_0^T re^{-rt}\ind_{[0,T]\setminus U^n_T}(t)\rmd t\right]\leq\ep.\]
Set $\widehat\theta=(\widehat t_k)_{k\in\dN}$, with $\widehat t_k=t^{\widehat n}_k$ for every $k \in \dN$, and $\widehat U_T=U^{\widehat n}_T$.
	
Define a map $\sigma^{\rm Bard} : \calU^{\rm Aqua} \to \calU^{\rm Bard}$ as follows:
for every $u\in\calU^{\rm Aqua}$, every $k\in\dN$, and every $t\in [\widehat t_k,\widehat t_{k+1})$, 
set
\[ \sigma^{\rm Bard}(u)_t=\alpha_{\ep}\left(\widehat u\right)_{\frac{t-\widehat t_k}{\widehat t_{k+1}-\widehat t_k}},\]
where $\widehat u_s=u_{\widehat t_k+(\widehat t_{k+1}-\widehat t_k)s}$, $s\in[0,1]$.
Here two successive time changes permit first to transform the control $u$ on the interval $[\widehat t_k,\widehat t_{k+1}]$ to a control on $[0,1]$ so that the strategy $\alpha_\ep$ can handle it, and then to retransform the output $\alpha_\ep\left(\widehat u\right)$ back to a control on the interval $[\widehat t_k,\widehat t_{k+1}]$.
One can verify that $\sigma^{\rm Bard}\in\calS^{\rm Bard}$ is a non-anticipative strategy with delay for Bard.

By the choice of $T$ and the definition of $\widehat\theta$, we have
\[ \begin{array}{rl}
	\E_{\mu^{\rm Aqua},\sigma^{\rm Bard}}\left[\int_0^\infty re^{-rt}\ind_{\{ u^{\rm Aqua}_t=u^{\rm Bard}_t\} }\rmd t\right]\geq &
		1-\ep -\E_{\mu^{\rm Aqua},\sigma^{\rm Bard}}\left[\int_0^T re^{-rt}\ind_{\{ u^{\rm Aqua}_t\neq u^{\rm Bard}_\}}\rmd t\right]\\
		\geq & 1-\ep
	-\E_{\mu^{\rm Aqua},\sigma^{\rm Bard}}\left[\int_0^T re^{-rt}\ind_{\{ u^{\rm Aqua}_t\neq u^{\rm Bard}_t\}}\ind_{\widehat U_T}(t)\rmd t\right].
	\end{array}
		 \]
The condition $(\widehat t_k,\widehat t_{k+1})\cap\theta=\emptyset$ in the definition of $\widehat U_T$ implies that, on $(\widehat t_k,\widehat t_{k+1})$, $\prob^{\rm Aqua}$-a.s., the control $u^{\rm Aqua}$ depends on $\sigma^{\widehat \theta}_\ep$ only through its restriction to $[0,\widehat t_k)$. In particular, Aqua cannot modify its control during the time interval $(\widehat t_k,\widehat t_{k+1})$, so that  Proposition \ref{minmax for choosing a random} applies: 
by the definition of $\sigma^{\rm Bard}$, we have, $\prob_{\mu^{\rm Aqua},\sigma^{\rm Bard}}$-a.s., on each interval $(\widehat t_k,\widehat t_{k+1})$ included in $\widehat U_T$, 
\[ \int_{\widehat t_k}^{\widehat t_{k+1}}re^{-rt}\ind_{\{ u^{\rm Aqua}_t\neq u^{\rm Bard}_t\}}\rmd t\leq r2^{-\widehat n}\int_0^1 \ind_{\left\{ u^{\rm Aqua}_{(\widehat t_k+(\widehat t_{k+1}-\widehat t_k)s)}\neq u^{\rm Bard}_{\widehat t_k+(\widehat t_{k+1}-\widehat t_k)s}\right\}}\rmd s.\]
Thus
\[
	\E_{\mu^{\rm Aqua},\sigma^{\rm Bard}}\left[\int_0^T re^{-rt}\ind_{\{ u^{\rm Aqua}_t\neq u^{\rm Bard}_t\}}\ind_{\widehat U_T}(t)\rmd t\right]
\leq  r2^{-n_*}\sum_{k=0}^{2^{\widehat n}-1}\E^{\rm Aqua}\left[\int_0^1 \ind_{\left\{ \widetilde u_s=\alpha_\ep(\widetilde u)_{\widehat t_k+s(\widehat t_{k+1}-\widehat t_k)}\right\}}\rmd s\right],\]
with $\widetilde u_s=u^{\rm Aqua}_{\widehat t_k+(\widehat t_{k+1}-\widehat t_k)s}$.
By the definition of $\alpha_\ep$, the last term is smaller that $r\ep$.
The result follows.
	\end{proof}

\end{document}